\documentclass[10pt]{article}

\usepackage{amsmath}
\usepackage{amsthm}
\usepackage{amssymb}
\usepackage{amsfonts}
\usepackage{graphicx}
\usepackage{verbatim}
\usepackage{mathrsfs}
\usepackage{subfigure}
\usepackage{fancyhdr}
\usepackage{latexsym}
\usepackage{graphicx}
\usepackage{dsfont}
\usepackage{epsf} 
\usepackage{color}

\theoremstyle{plain}
\newtheorem{theorem}{Theorem}[section]
\newtheorem{proposition}[theorem]{Proposition}
\newtheorem{lemma}[theorem]{Lemma}

\newtheorem{remark}[theorem]{Remark}
\newtheorem{definition}[theorem]{Definition}

\newcommand{\R}{\mathbb{R}}
\newcommand{\N}{\mathbb{N}}
\newcommand{\C}{\mathcal{C}}
\newcommand{\ds}{\displaystyle}
\DeclareMathOperator{\diverg}{div}
\DeclareMathOperator{\transp}{\!{}^T}

\newcommand{\constC}{\mathcal{K}}
\newcommand{\constdel}{\delta}
\newcommand{\constD}{D}
\newcommand{\eps}{\varepsilon}

\author{Maxime Breden \thanks{CMLA, ENS Cachan, CNRS, Universit\'e Paris-Saclay, 61 avenue du Pr\'esident Wilson, 94235 Cachan, France. {\tt breden@cmla.ens-cachan.fr}}
\and Laurent Desvillettes \thanks{Univ. Paris Diderot, Sorbonne Paris Cit\'e, Institut de Math\'ematiques de Jussieu - Paris Rive Gauche, UMR 7586, CNRS, Sorbonne Universit\'es, UPMC Univ. Paris 06, F-75013, Paris, France. {\tt desvillettes@math.univ-paris-diderot.fr}}
\and Klemens Fellner \thanks{Institute of Mathematics and Scientific Computing, NAWI Graz, University of Graz, Heinrichstr. 36, 8010 Graz, Austria. {\tt klemens.fellner@uni-graz.at}}}

\begin{document}

\title{Smoothness of moments of the solutions of
 discrete coagulation equations with diffusion}

\date{\today}

\maketitle

\begin{abstract}
In this paper, we establish smoothness of moments of the solutions of
discrete coagulation-diffusion systems. As key assumptions, we suppose that
the coagulation coefficients grow at most sub-linearly 
and that the diffusion coefficients  converge towards a strictly positive limit
(those conditions also imply the existence of global weak solutions and the absence of gelation).
\end{abstract}

\begin{center}

MSC subject class: 35B45, 35B65, 82D60
\end{center}

\section{Introduction}\label{sec:intro}

In this  paper  we consider discrete coagulation systems with spatial diffusion. Coagulation models appear in a wide range of applications ranging from chemistry (e.g. the formation of polymers) over physics (aerosols, raindrops, smoke, sprays), 
astronomy (the formation of galaxies) to biology (haematology, animal grouping), see e.g. the surveys \cite{Dra,LM04,DeFe13} and the references therein. 

Following the pioneering works of Smoluchowski (see \cite {Smo16,Smo17}), we shall denote by $c_i :=c_i(t,x) \in \R_+$ the concentration of polymers or clusters of mass/size $i\in\N^*$ at time $t$ and position $x$. Here, we consider a smooth bounded domain $\Omega$ of $\R^N$ in which the clusters are confined via homogeneous Neumann conditions (in applications, we of course have $N \le 3$).
Moreover, for any positive time $T$, we denote by $\Omega_T$ the set $[0,T] \times\Omega$. 
\medskip

We assume that the concentrations $c_i$ satisfy the following infinite (for all $i\in\N^* := \N \setminus \{0\}$) set of reaction-diffusion equations
with homogeneous Neumann boundary conditions:
\begin{equation}
\label{eq:syst_coag-frag}
\left\{
\begin{aligned}
& \partial_t c_i - d_i \Delta_x c_i = Q_i(c) \quad &\text{on}& \ \Omega_T, \\
& \nabla_x c_i\cdot \nu = 0 \quad &\text{on}& \ [0,T]\times\partial\Omega, \\
& c_i(0,\cdot) = c_i^{in} \quad &\text{on}& \ \Omega,
\end{aligned}
\right.
\end{equation}
where $d_i>0$ are strictly positive diffusion coefficients, $\nu(x)$ denotes the outward unit normal vector at point $x\in\partial\Omega$ and $c_i^{in}$ are given initial data, which are typically assumed nonnegative. 
\par
 The coagulation terms $Q_i(c)$ depend on the whole sequence of concentrations $c=\left(c_i\right)_{i\in \N^*}$ and
can be written as the difference between a gain term ¤$Q_i^+(c)$ and a loss term $Q_i^-(c)$, 
which take the form 
\begin{equation} \label{eq:coag}
Q_i(c) := Q_i^+(c) - Q_i^-(c) = \frac{1}{2}\sum_{j=1}^{i-1}a_{i-j,j}c_{i-j}c_j - \sum_{j=1}^{\infty}a_{i,j}c_ic_j .
\end{equation}
Here, the nonnegative parameters 
 $a_{i,j}$ represent
the coagulation coefficients
of clusters of size $i$ merging with clusters of size $j$, which are symmetric: $a_{i,j}=a_{j,i}$. In this work, we consider the case in which the coagulation 
 coefficients additionally satisfy the following asymptotic behaviour:
\begin{equation}
\label{hyp:exist_LM}
\lim\limits_{j\to\infty} \frac{a_{i,j}}{j} = 0, \quad \forall~i\in\N^*,\qquad\quad 0\le a_{i,j}=a_{j,i}, \quad \forall~i,j\in\N^*.
\end{equation}
The conditions \eqref{hyp:exist_LM} are sufficient to provide the existence of global $L^1$-weak solutions (with
nonnegative concentrations) to system \eqref{eq:syst_coag-frag}-\eqref{eq:coag} (for which the below estimate~\eqref{eq:ineg_masse}  on the mass holds), when suitable nonnegative initial data are considered, see \cite{LauMis02}.

Thanks to the symmetry assumption on the coagulation coefficients, we can write (at a formal level) the following weak formulation of the coagulation operator: For any test-sequence $(\varphi_i)_{i \in \N^*}$, we have
\begin{equation} 
\label{wc}
\sum_{i=1}^{\infty} \varphi_i\,Q_i(c) = \frac{1}{2}\sum_{i=1}^{\infty}\sum_{j=1}^{\infty}a_{i,j}\, c_i\, c_j\,(\varphi_{i+j}-\varphi_i-\varphi_j).
\end{equation}

In the sequel, we shall systematically denote 
for any $k\in\R_+$  by
\begin{equation} \label{nd}
\rho_k(t,x):=\sum_{i=1}^{\infty}i^k\, c_i(t,x)
\end{equation}
the moment of order $k$ of the sequence of concentrations $(c_i)_{i\in \N^*}$ 
and, similarly, the  moment of order $k$ of the initial concentrations by
\begin{equation} \label{ndi}
\rho_k^{in}(x) := \sum_{i=1}^{\infty}i^k\,c_i^{in}(x).
\end{equation}

By taking $\varphi_i = i$ in (\ref{wc}), we see that (still at a formal level) the conservation
of the total mass contained in all clusters/polymers holds, that is, 
\begin{equation} \label{consmasse}
\forall t\ge 0, \qquad \int_{\Omega} \rho_1(t,x)\, dx = \int_{\Omega} \sum_{i=1}^{\infty}i\, c_i(t,x)\,dx = \int_{\Omega} \sum_{i=1}^{\infty}i\,c_i^{in}(x)\,dx =  \int_{\Omega} \rho_1^{in}(x)\, dx.
\end{equation}
It is a well-known phenomenon for coagulation models, called {\em gelation}, that  the formal conservation of the total mass \eqref{consmasse} will not hold for solutions of coagulation models with sufficiently growing coagulation coefficients $a_{i,j}$ (already for space homogeneous models): When approximating the first order moment $\rho_1(t,x)$ as the cut-off limit $ \lim_{K\to\infty} \sum_{i=1}^{\infty} \min\{i,K\} c_i$, then the weak formulation \eqref{wc} for the test-sequence $\varphi_i=\min\{i,K\}$
shows that the map $t\mapsto\sum_{i=1}^{\infty} \min\{i,K\} c_i$ is non-increasing in time, and Fatou's
lemma only implies that the total mass is non-increasing in time (for space homogeneous and space inhomogeneous coagulation with homogeneous Neumann boundary conditions models alike). The conservation law \eqref{consmasse} can become a strict inequality
for solutions of (\ref{eq:syst_coag-frag}) with sufficiently growing coagulation coefficients, but we  still get a natural uniform-in-time bound in $L^{\infty} (\R_+ ; L^1(\Omega))$ of the total mass $\rho_1$, namely
\begin{equation}\label{eq:ineg_masse}
\forall t\ge 0, \qquad \int_{\Omega} \rho_1(t,x)\, dx \leq  \int_{\Omega} \rho_1^{in}(x)\, dx.
\end{equation}
\medskip 

A standard way to prove the existence of weak solutions to system \eqref{eq:syst_coag-frag}-\eqref{eq:coag} satisfying~\eqref{eq:ineg_masse} is to consider a sequence of truncated systems for which we can prove existence of smooth solutions. Then, using some compactness arguments, one extracts a solution of the limiting system \eqref{eq:syst_coag-frag}-\eqref{eq:coag}, again see for instance \cite{LauMis02}. In this work, for any $n\in\N^*$, we define $c^n=\left(c_1^n,\ldots,c_n^n\right)$ as the solution of the truncated problem:
$\forall~ 1\leq i\leq n$,
\begin{align}\label{systtrunc}
\left\{
\begin{aligned}
& \partial_t c_i^n - d_i \Delta_x c_i^n = Q_i^n(c^n) \quad &\text{on}& \ \Omega_T, \\
& \nabla_x c_i^n\cdot \nu = 0 \quad &\text{on}& \ [0,T]\times\partial\Omega, \\
& c_i^n(0,\cdot) = c_i^{in} \quad &\text{on}& \ \Omega,
\end{aligned}
\right.
\end{align}
where
\begin{equation} \label{qtrunc}
Q_i^n(c^n) = \frac{1}{2}\sum_{j=1}^{i-1}a_{i-j,j}c_{i-j}^nc_j^n-\sum_{j=1}^{n-i}a_{i,j}c_i^nc_j^n .
\end{equation}
This is now a finite system of reaction-diffusion equations with finite sums in the r.h.s, for which the existence and uniqueness of nonnegative, global and smooth solutions is classical (see for example Proposition 2.1 and Lemma 2.2 of \cite{Wrz97}, or \cite{DMilan}). Notice that for any sequence $(\varphi_i)_{i \in \N^*}$, we have
\begin{equation} 
\label{wc_truncated}
\sum_{i=1}^{n} \varphi_i\,Q_i^n(c^n) = \frac{1}{2}\sum_{i+j\leq n; i,j \ge 1} a_{i,j}\, c_i^n\, c_j^n\,(\varphi_{i+j}-\varphi_i-\varphi_j),
\end{equation}
so that we get (this time rigorously)
\begin{equation*}
\forall t\ge 0, \qquad \int\limits_{\Omega} \sum_{i=1}^n ic_i^n(t,x)\, dx =  \int\limits_{\Omega} \sum_{i=1}^n ic_i^{in}(x)\, dx .
\end{equation*}
If we manage to extract a limit from $(c_n)$, Fatou's Lemma then
yields~\eqref{eq:ineg_masse} for the limiting concentration.

\medskip

Before proceeding further, let us introduce a precise definition of weak solution, following~\cite{LauMis02}.
\begin{definition}\label{def:sol_faible}
A global weak solution $c=\left(c_i\right)_{i\in\N^*}$ to \eqref{eq:syst_coag-frag}-\eqref{eq:coag} is a sequence of functions $c_i:[0,+\infty)\times\Omega\to[0,+\infty)$ such that, for all $i\in\N^*$ and $T>0$
\begin{itemize}
\item $c_i\in \mathcal{C}\left([0,T];L^1(\Omega)\right)$,
\item $Q^-_i(c)\in L^1(\Omega_T)$,
\item $\sup\limits_{t\geq 0}\int\limits_{\Omega}\rho_1(t,x)dx \leq \int\limits_{\Omega}\rho_1^{in}(x)dx$,
\item $c_i$ is a mild solution to the $i$-th equation in~\eqref{eq:syst_coag-frag}, that is
\begin{equation*}
c_i(t)=e^{d_iA_1t}c_i^{in} + \int_0^t e^{d_iA_1(t-s)}Q_i(c(s))ds,
\end{equation*}
where $Q_i$ is defined by~\eqref{eq:coag}, $A_1$ is the closure in $L^1(\Omega)$ of the unbounded, linear, self-adjoint operator $A$ of $L^2(\Omega)$ defined by
\begin{equation*}
D(A)=\left\{w\in H^2(\Omega),\ \nabla w\cdot \nu=0 \emph{ on }\partial\Omega \right\}, \qquad Aw=\Delta w,
\end{equation*}
and $e^{d_iA_1t}$ is the $\mathcal{C}^0$-semigroup generated by $d_iA_1$ in $L^{1}(\Omega)$.
\end{itemize}
\end{definition}

The following result, which is a direct application of~\cite[Theorem 3]{LauMis02}, states that we can obtain weak solution of~\eqref{eq:syst_coag-frag}-\eqref{eq:coag} from the truncated systems~\eqref{systtrunc}-\eqref{qtrunc}. We also refer to \cite{Wrz97} and \cite{Wrz}.

\begin{proposition}
\label{prop:extraction}
Let $\Omega$ be a smooth bounded domain of $\R^N$. Assume that the coagulation coefficients satisfy~\eqref{hyp:exist_LM} and that  all diffusion coefficients are strictly positive, i.e. $d_i>0$ $\forall~i\in\N^*$. Assume also that the initial concentrations $c_i^{in}\geq 0$ are such that $\rho_1^{in}\in L^1(\Omega)$. For every $n\in\N^*$, let $c^n=(c_1^n,\ldots,c_n^n)$ be the solution of the truncated system of size $n$~\eqref{systtrunc}-\eqref{qtrunc}.
\par
Then, there exists a sequence $c=(c_i)_{i\in\N^*}$ such that, up to extraction
\begin{equation*}
c_i^n\underset{n\to\infty}{\longrightarrow} c_i \quad \text{in }L^1(\Omega_T),\quad \forall~i\in\N^*,\ \forall~T>0,
\end{equation*}
and $c$ is a weak solution to~\eqref{eq:syst_coag-frag}-\eqref{eq:coag} in the sense of Definition \ref{def:sol_faible}.
\end{proposition}

\medskip 

Our first proposition states that if the diffusion rates of clusters of different sizes are sufficiently close to each others, the natural uniform $L^1$-bound~\eqref{eq:ineg_masse} can be extended to $L^p$ (with $p>1$ depending on the closeness of the diffusion rates). To be more precise about this closeness hypothesis, let us first introduce 

\begin{definition}
\label{def:C_mq}
For $m>0$ and $q\in]1,+\infty[$, we define $\constC_{m,q} > 0$ as the best (i.e. the smallest) constant independent of $T>0$ in the parabolic regularity estimate 
\begin{equation*}
\left(\,\int_{\Omega_T} \left\vert \partial_t v \right\vert^q +m^q \int_{\Omega_T} \left\vert \Delta_x v \right\vert^q\right)^{\frac{1}{q}} \leq \constC_{m,q} \left(\,\int_{\Omega_T} \left\vert f \right\vert^q\right)^{\frac{1}{q}},
\end{equation*}
where $v$ is the unique solution of the heat equation with constant diffusion coefficient $m$, homogeneous Neumann boundary conditions and zero initial data:
\begin{equation*}
\left\{
\begin{aligned}
& \partial_t v - m\Delta_x v = f \quad &\text{on}& \ \Omega_T, \\
& \nabla_x v\cdot \nu = 0 \quad &\text{on}& \ [0,T]\times\partial\Omega, \\
&  v(0,\cdot) = 0 \quad &\text{on}& \ \Omega.
\end{aligned}
\right.
\end{equation*}
\end{definition}
The existence of a such a constant $\constC_{m,q}<\infty$ independent of the
 time $T>0$ is explicitly stated in \cite{Lam87} provided that $\partial\Omega \in \C^{2+\alpha}, \alpha>0$.
\medskip

Next, we present the

\begin{proposition}\label{th:first_moment}
Let $\Omega$ be a smooth bounded domain of $\R^N$ (e.g. $\partial\Omega \in \C^{2+\alpha}, \alpha>0$).  Let $p\in]1,+\infty[$ and assume that the nonnegative initial data $c_i^{in}\ge 0$ have an initial mass $\rho_1^{in}$ which lies in $L^p(\Omega)$.
Assume that the coagulation coefficients satisfy
\eqref{hyp:exist_LM}.
Assume that
\begin{equation}\label{deltaD}
0 < \constdel := \inf_{i\geq 1} d_i, \qquad\text{and}\qquad
\constD := \sup_{i\geq 1} d_i < \infty.
\end{equation}

Then, provided that for the H\"older conjugate $p'$ of $p$
holds the condition 
\begin{equation}
\label{hyp:closeness}
\frac{\constD- \constdel}{\constD+\constdel}\, \constC_{\frac{\constD+\constdel}{2},p'}<1,
\end{equation}
there exists a weak solution of the coagulation system \eqref{eq:syst_coag-frag}-\eqref{eq:coag} such that the mass $\rho_1$ lies in $L^p(\Omega_T)$ for any finite time $T>0$.
\end{proposition}

\begin{remark}\label{dualL2eps}
Note that the above estimate was already proven in \cite{CanDevFel10} in the particular case $p=2$, even without assuming~\eqref{hyp:closeness}. 
In fact, the Hilbert space case $p=2$ allows to prove the explicit bound $\constC_{m,2}\leq 1$ (see Lemma~\ref{lem:q=2}), which leads to
\begin{equation}\label{hyp:closenessL2}
\frac{\constD- \constdel}{\constD+\constdel}\constC_{\frac{\constD+\constdel}{2},2}
\leq \frac{\constD- \constdel}{\constD+\constdel} 
<1,
\end{equation}
and hypothesis~\eqref{hyp:closeness} is automatically satisfied for $p=2$ for all $0<\constdel\le\constD<\infty$ and $T>0$
(hence its absence in \cite{CanDevFel10}). 

Note that this global $L^2$-bound together with assumptions \eqref{hyp:exist_LM} also ensures that no gelation can occur, so that the conservation law \eqref{consmasse} rigorously holds for any weak solution, see \cite{CanDevFel10}.

Moreover, the strict inequality in \eqref{hyp:closenessL2} has been further exploited in \cite{CanDevFel13} by proving a continuous upper bound of the best constant $\constC_{m,p'}$ on $p'\le2$. Therefore, for all $0<\constdel\le\constD<\infty$,
 there exists 
a sufficiently small $0<\eps=\eps(\delta,D)\ll1$ such that \eqref{hyp:closenessL2} can be slightly improved to 
$$
\frac{\constD- \constdel}{\constD+\constdel}\, \constC_{\frac{\constD+\constdel}{2},2-O(\eps)}<1,
$$ 
and this allows to prove a correspondingly improved a priori estimate in $L^{2+\eps}(\Omega_T)$.
\end{remark}
\medskip

Proposition \ref{th:first_moment} can be improved in the case when the diffusion coefficients $\left(d_i\right)_{i\in \N^*}$ constitute a sequence converging towards a strictly positive limit. Note that such an assumption is not so far from the assumption that the sequence $\left(d_i\right)_{i\in \N^*}$ is bounded above and below (by a strictly positive constant), which
 is used in Proposition \ref{th:first_moment} 
(or also in \cite{CanDevFel10}),
 since one expects on physical grounds that the sequence  $\left(d_i\right)_{i\in \N^*}$ is decreasing; that is, that larger clusters diffuse less. Under this assumption and provided that the coagulation coefficients are strictly sublinear (see the precise assumption in Proposition \ref{th:all_moments} below)
 we can show that $L^p$ norms of moments $\rho_k$ are propagated for any $k \in \N^*$, $p \in ]1, \infty[$. 

\begin{proposition}\label{th:all_moments}
Let $\Omega$ be a smooth bounded domain of $\R^N$.
Assume that the coagulation coefficients satisfy for a constant $C>0$ and all $i,j\in\N^*$
\begin{equation}\label{sub}
0\le a_{i,j}= a_{j,i}\leq C\, (i^{\gamma}+j^{\gamma}), \quad \text{for some  } \gamma\in[0,1[,  
\end{equation}
 and that $\left(d_i\right)_{i \in \N}$ is a sequence of strictly positive real numbers which converges toward a strictly positive limit.
 \par
 Assume that (for some $k\in\N^*$)
 the initial moment $\rho_k^{in}$ lies in $L^p(\Omega)$ for 
all $p<+\infty$ and that (for all $i\in \N^*$) each initial concentration $c_i^{in} \ge 0$ lies in $L^{\infty}(\Omega)$.
\par
 Then, there exists a global weak nonnegative solution to~\eqref{eq:syst_coag-frag}-\eqref{eq:coag} for which the moment $\rho_k$ lie in $L^p(\Omega_T)$ for all $p<+\infty$ and all finite time  $T>0$.
\end{proposition}

\begin{remark}
Notice that hypothesis~\eqref{sub} on the coagulation coefficients implies the assumption \eqref{hyp:exist_LM}, which in return yields existence of global weak solutions. 
 \par
In the existing literature, sublinear assumptions on the coagulation coefficients are often found under the form:
\begin{equation}\label{sub_bis}
0\le a_{i,j}= a_{j,i}\leq \tilde C\, (i^{\alpha}j^{\beta}+i^{\beta}j^{\alpha}), \quad \emph{for some }\alpha,\beta\in[0,1[ \emph{ with } \alpha+\beta<1. 
\end{equation} 
Our motive for using assumption~\eqref{sub} rather than assumption~\eqref{sub_bis} is mainly that it allows for slightly shorter computations, without any loss of generality since \eqref{sub_bis} implies \eqref{sub} with $\gamma=\alpha+\beta$.
 \par
Finally, we point out that Proposition~\ref{th:all_moments} could be extended to the case where a finite number of diffusion coefficients $d_i$ are equal to $0$ (see Remark~\ref{remark:d_i=0}).
\end{remark}

\medskip

Results in the same spirit about propagation of moments have been obtained recently in \cite{Rez10} and \cite{Rez14}, where the system~\eqref{eq:syst_coag-frag}-\eqref{eq:coag} and its continuous counterpart are studied on the whole space $\R^N$. Assuming~\eqref{deltaD} and a finite total increase of variation for $(d_i)$, together with a control on the growth of the coagulation coefficients (also involving the diffusion rates) such as
\begin{equation*}
\frac{a_{i,j}}{(i+j)(d_i+d_j)}\xrightarrow[i+j\to+\infty]{}0, 
\end{equation*}
and the finiteness of some of the initial moments (in different norms), bounds 
are obtained which look like
\begin{equation*}
\left\Vert \rho_k(t,\cdot)\right\Vert_{L^p(\Omega)} \leq \left\Vert \rho_k^{in}\right\Vert_{L^p(\Omega)} + Ck^{-l},
\end{equation*}
where $l$ depends on the degree of the initial moments assumed to be finite. 
\medskip

The statement of our result is therefore close to that of \cite{Rez14} (our requirement on the diffusion rate is however more stringent), but the proof is 
completly different, so that the exact conditions required on initial data are also different.
 Note that the limit case $a_{i,j} = i + j$ is still open (absence of gelation for this
 coagulation coefficient is conjectured in general, but is proven only when there is no diffusion).
 \medskip
 
 Although in the present work, $L^p$ estimates for moments are only shown for $p < \infty$ (whereas $p= \infty$ can be obtained in \cite{Rez14}), the use of parabolic inequalities for the heat equation enables to recover this case (and also higher order derivatives).  
   \medskip
   
Indeed, the estimates obtained in Proposition~\ref{th:all_moments} can be improved if the initial data are assumed to be smooth enough. This leads to our main Theorem, namely:
 
\begin{theorem}\label{th:all_moments_w}
Let $\Omega$ be a smooth bounded domain of $\R^N$. 
Assume that the coagulation coefficients satisfy (\ref{sub}) and that $\left(d_i\right)_{i \in \N}$ is a sequence of strictly positive real numbers which converges toward a strictly positive limit. 
\par
Assume that the initial data $c_i^{in} \ge 0$ are of class $\C^{\infty}(\overline \Omega)$, compatible with the boundary conditions, and that for all $k \in \N^*$ the initial moments $\rho_k^{in}$ are of class $\C^{\infty}(\overline \Omega)$.
\par
 Then, there exist a unique smooth solution to~\eqref{eq:syst_coag-frag}-\eqref{eq:coag} such that each $(c_i)$ is nonnegative, of class $\C^{\infty}(\overline\Omega_T)$ for any finite time $T>0$, and such that the moments $\rho_k$ are also of class $\C^{\infty}(\overline\Omega_T)$, for any $k\in\N^*$.
\end{theorem}

\begin{remark}
The $\C^{\infty}$ regularity down to time $0$ requires of course the $\C^{\infty}$ hypothesis on the initial data. However, it can be seen in the various steps of the proof (see Section~\ref{sec:max_reg}) that propagation of regularity
 in intermediate Sobolev spaces holds
  under suitable (less stringent) assumptions on the initial data. 
 \par
 Since each $c_i$ is solution of a heat equation subject to a r.h.s. that can be controlled once all moments are bounded in $L^p(\Omega_T)$, $p<+\infty$, we can in fact show the creation of regularity for strictly positive times. For example, under the assumption 
that $\rho_k^{in}\in L^p(\Omega)$ for all $p<+\infty$ and all $k\in\mathbb{N}^*$, we can prove that the concentrations $c_i$  are of class $C^{\infty}(]0,T] \times \bar{\Omega})$.
 \par
 Also, as will be made clear in Section~\ref{sec:max_reg},
  $\C^{\infty}$ regularity is not needed to ensure uniqueness. As shown in \cite{HamRez07}, uniqueness holds as soon as $\rho_2\in L^{\infty}$, so that starting from initial data leading to an estimate for 
$\rho_2$  in a Sobolev space embedded in $L^{\infty}$, uniqueness can already
be obtained.
 \par
Finally, we point out that assumption (\ref{sub}) is not far from optimal, since it is known that gelation can occur 
as soon as $a_{i,j} =  i^{\alpha}j^{\beta}+i^{\beta}j^{\alpha}$ with $\alpha+\beta>1$ 
(see \cite{EscMisPer02}) and gelation is not compatible with the conclusion of Proposition~\ref{th:all_moments} or Theorem~\ref{th:all_moments_w}.  
\end{remark}
  
 \medskip
 
 
 
Our paper is organized as follows. In Section~\ref{sec:duality}, we recall some
 lemmas existing in the literature and called duality lemmas. We also introduce modified versions of those lemmas, that are later used in Section~\ref{sec:prop_Lp} to prove the propagation of moments in $L^p(\Omega_T)$, $p<+\infty$ (Propositions~\ref{th:first_moment} and \ref{th:all_moments}). In Section~\ref{sec:max_reg}, we extend these results  to prove $\C^{\infty}$ regularity for the concentrations and the moments (Theorem~\ref{th:all_moments_w}). Finally, a short Appendix is devoted to technical lemmas which are useful to make the proof of some duality lemmas rigorous.

\section{Duality estimates}
\label{sec:duality}

We start by recalling some {\it {a priori}} estimates based on duality arguments from \cite{CanDevFel13}. These estimates are key ingredients of the present work. In this section, functions said to be weak solutions ought to be understood as solutions of the equation obtained by multiplying by a test function and integrating by parts. Remember also that $\constC_{m,q}$ is defined in Definition \ref{def:C_mq}.
\medskip

The first statement recalls \cite[Lemma~2.2]{CanDevFel13}.

\begin{lemma}\label{lem:parabolic_estimate}
Let $\Omega$ be a smooth bounded subset of $\R^N$ and
consider a function $M:= M(t,x): \Omega_T\to\R_+$ satisfying $a\leq M \leq b$ for some $a,b>0$. For any $q\in]1,+\infty[$, if 
\begin{equation}
\label{hyp:closeness_bis}
\frac{b-a}{b+a}\,\constC_{\frac{a+b}{2},q}<1,
\end{equation}
then, there exist constants $C_0>0$ and $C>0$ (depending on $\Omega, a,b,q,T$) such that for any $f\in L^q(\Omega)$, the (unique, weak) solution $v$ of the backward parabolic system with $L^{\infty}$ coefficient $M:=M(t,x)$,
\begin{equation*}
\left\{
\begin{aligned}
& \partial_t v + M\Delta_x v = f \quad &\text{on}& \ \Omega_T, \\
& \nabla_x v\cdot \nu = 0 \quad &\text{on}& \ [0,T]\times\partial\Omega, \\
&  v(T,\cdot) = 0 \quad &\text{on}& \ \Omega,
\end{aligned}
\right.
\end{equation*}
satisfies  $\left\Vert v\right\Vert_{L^q(\Omega_T)} \leq C \left\Vert f\right\Vert_{L^q(\Omega_T)}$ and $\left\Vert v(0, \cdot)\right\Vert_{L^q(\Omega)} \leq C_0 \left\Vert f\right\Vert_{L^q(\Omega_T)}$.
\end{lemma}

\begin{remark}
The bound on $\left\Vert v\right\Vert_{L^q(\Omega_T)}$ is not explicitly mentioned in Lemma~2.2 of \cite{CanDevFel13}, but is a direct consequence of its proof, in particular of the estimates $\left\Vert \Delta_x v\right\Vert_{L^q(\Omega_T)} \le C_1 \left\Vert f\right\Vert_{L^q(\Omega_T)}$ and $\left\Vert \partial_t v\right\Vert_{L^q(\Omega_T)} \le C_1 \left\Vert f\right\Vert_{L^q(\Omega_T)}$, which are explicitly mentioned there.
\end{remark}

\begin{remark}
The fact that the above mentioned function $v$ exists (and is unique) is (in particular for $q<2$) not 
obvious because $M$ is not assumed to be continuous (or at least VMO). 
In the Appendix (Proposition~\ref{prop:exist_para_disc}), we give a proof of the existence and uniqueness of $v$ for the sake of completeness.
\end{remark}

Lemma~\ref{lem:parabolic_estimate} is used to prove the following duality lemma, which is Proposition~1.1 of \cite{CanDevFel13}.
\begin{proposition}\label{prop:duality}
Let $\Omega$ be a smooth bounded subset of $\R^N$ and
consider a function $M:=M(t,x) :\Omega_T\to\R_+$
 satisfying $a\leq M \leq b$ for some $a,b>0$. For any $p\in]1,+\infty[$, if 
\begin{equation*}
\frac{b-a}{b+a}\,\constC_{\frac{a+b}{2},p'}<1,
\end{equation*}
then, there exists a constant $C>0$ (depending on $\Omega$, $a,b,p,T$) such that for any $u_0\in L^p(\Omega)$, any weak solution $u$ of the parabolic system (in divergence form)
\begin{equation*}
\left\{
\begin{aligned}
& \partial_t u - \Delta_x \left(Mu\right) = 0 \quad &\text{on}& \ \Omega_T, \\
& \nabla_x u\cdot \nu = 0 \quad &\text{on}& \ [0,T]\times\partial\Omega, \\
&  u(0,\cdot) = u_0 \quad &\text{on}& \ \Omega,
\end{aligned}
\right.
\end{equation*}
satisfies  $\left\Vert u\right\Vert_{L^p(\Omega_T)} \leq C \left\Vert u_0\right\Vert_{L^p(\Omega)}$.
\end{proposition}

In the sequel we will need a generalized version of Proposition~\ref{prop:duality}, which is an adaptation of Theorem~3.1 in \cite{DevFelPieVov07} (where only the case $p=2$ is treated).
\begin{proposition}\label{prop:duality_rhs}
Let $\Omega$ be a smooth bounded subset of $\R^N$, $\mu_1,\mu_2\ge 0$, and
consider a function $M:=M(t,x): \Omega_T\to\R_+$ satisfying $a\leq M \leq b$ for some $a,b>0$. 
 For any $p\in]1,+\infty[$, if 
\begin{equation}
\label{eq:hyp_C}
\frac{b-a}{b+a}\,\constC_{\frac{a+b}{2},p'}<1,
\end{equation}
then, there exists a constant $C>0$ (depending on $\Omega, a,b,p, \mu_1,\mu_2,T$) such that for any $u_0\in L^p(\Omega)$, any function $u:\Omega_T\to \R_+$ satisfying (weakly)
\begin{equation}
\label{eq:duality_2}
\left\{
\begin{aligned}
& \partial_t u - \Delta_x\left(Mu\right) \leq \mu_1 u + \mu_2 \quad &\text{on}& \ \Omega_T,\\
& \nabla_x u\cdot \nu = 0 \quad &\text{on}& \ [0,T]\times\partial\Omega, \\
& u(0,\cdot) = u_0 \quad &\text{on}& \ \Omega,
\end{aligned}
\right.
\end{equation}
belongs  to $L^p(\Omega_T)$, with the estimate:
\begin{equation*}
\left\Vert u\right\Vert_{L^p(\Omega_T)} \le C \left(1+\left\Vert u_0\right\Vert_{L^p(\Omega)}\right).
\end{equation*}
\end{proposition}
\begin{proof}
Let $\varphi$ be a nonnegative smooth function on $\Omega_T$
 and $v$ be the (unique, weak) solution (cf. Prop. \ref{prop:exist_para_disc}) of the dual problem
 
\begin{equation*}
\left\{
\begin{aligned}
& \partial_t v + M\Delta_x v + \mu_1 v  = -\varphi \quad &\text{on}& \ \Omega_T, \\
& \nabla_x v\cdot \nu = 0 \quad &\text{on}& \ [0,T]\times\partial\Omega, \\
& v(T,\cdot) = 0 \quad &\text{on}& \ \Omega.
\end{aligned}
\right.
\end{equation*}
Notice that the function $\tilde v$ defined by $\tilde v(t)=v(T-t)$ satisfies a standard, forward in time, reaction-diffusion equation 
\begin{equation*}
\left\{
\begin{aligned}
& \partial_t \tilde v - M\Delta_x \tilde v = \mu_1 \tilde v + \varphi \quad &\text{on}& \ \Omega_T,\\
& \nabla_x \tilde v\cdot \nu = 0 \quad &\text{on}& \ [0,T]\times\partial\Omega, \\
& \tilde v(0,\cdot) = 0 \quad &\text{on}& \ \Omega,
\end{aligned}
\right.
\end{equation*}
which ensure that $\tilde v$, and therefore $v$, are nonnegative. Multiplying \eqref{eq:duality_2} by the solution $v$ of the dual problem and integrating on $\Omega_T$, we end up with
\begin{equation}
\label{eq:dual_ineq}
\int\limits_{\Omega_T}u\,\varphi \leq \int\limits_{\Omega}u(0)v(0) + \mu_2\int\limits_{\Omega_T}v \leq \left\Vert u(0) \right\Vert_{L^{p}(\Omega)} \left\Vert v(0) \right\Vert_{L^{p'}(\Omega)} + \mu_2 \left(\vert\Omega\vert T\right)^{\frac{1}{p}} \left\Vert v \right\Vert_{L^{p'}(\Omega_T)}.
\end{equation}
Moreover, the rescaled function $w=e^{\mu_1 t}v$ satisfies
\begin{equation*}
\left\{
\begin{aligned}
& \partial_t w + M\Delta_x w = -e^{\mu_1 t}\varphi \quad &\text{on}& \ \Omega_T, \\
& \nabla_x w\cdot \nu = 0 \quad &\text{on}& \ [0,T]\times\partial\Omega, \\
& w(T, \cdot ) = 0 \quad &\text{on}& \ \Omega.
\end{aligned}
\right.
\end{equation*}
Thus, provided that hypothesis \eqref{eq:hyp_C} is satisfied, we can apply Lemma~\ref{lem:parabolic_estimate} to $w$ and get (after noticing that $\vert v\vert \leq \vert w\vert$)
\begin{equation*}
\left\Vert v\right\Vert_{L^{p'}(\Omega_T)} \le C \left\Vert \varphi\right\Vert_{L^{p'}(\Omega_T)}, \quad \text{and} \quad \left\Vert v(0)\right\Vert_{L^{p'}(\Omega)} \le C_0 \left\Vert \varphi\right\Vert_{L^{p'}(\Omega_T)},
\end{equation*}
where the term $e^{\mu_1 T}$ is absorbed in the constants. Returning to \eqref{eq:dual_ineq}, we finally obtain
\begin{equation*}
\int\limits_{\Omega_T}u\,\varphi \leq C \left(1+\left\Vert u_0\right\Vert_{L^p(\Omega)}\right)\left\Vert \varphi\right\Vert_{L^{p'}(\Omega_T)},
\end{equation*}
for all nonnegative smooth functions $\varphi$, and the statement of Proposition \ref{prop:duality_rhs} follows by duality.
\end{proof}


We finish this section with another variant of the duality lemma, in which $L^p$ r.h.s. can be treated.

\begin{proposition}
\label{prop:duality_eps}
Let $\Omega$ be a smooth bounded subset of $\R^N$  and
consider a function $M:= M(t,x) : \Omega_T\to\R_+$ satisfying $a\leq M \leq b$ for some $a,b>0$. Consider functions $A$ and $B$ defined on $\Omega_T$ and a real number $\varepsilon\in ]0,1[$. 
Assume that for some $p\in]1,+\infty[$, the following statements  hold: 
\begin{equation}
\label{eq:hyp_C3}
\frac{b-a}{b+a}\,\constC_{\frac{a+b}{2},p'}<1,\qquad A\in L^{\frac{p}{\varepsilon}}(\Omega_T),\quad
B\in L^p(\Omega_T).
\end{equation}
\par
 Then, there exists a constant $C$ (depending on $\Omega, a,b,p,\varepsilon, T$) 
such that for any $u_0\in L^p(\Omega)$, and any nonnegative $u\in L^p(\Omega_T)$ satisfying (weakly)
\begin{equation}
\label{eq:duality_eps}
\left\{
\begin{aligned}
& \partial_t u - \Delta_x (Mu) \leq A\,u^{1-\varepsilon}+B \quad &\text{on}& \ \Omega_T, \\
& \nabla_x u\cdot \nu = 0 \quad &\text{on}& \ [0,T]\times\partial\Omega, \\
& u(0,\cdot) = u_0 \quad &\text{on}& \ \Omega,
\end{aligned}
\right.
\end{equation}
the following estimate holds:
\begin{equation*}
\left\Vert u\right\Vert^p_{L^p(\Omega_T)} \le C \left(\left\Vert u_0\right\Vert^p_{L^p(\Omega)} + \left\Vert A\right\Vert^{\frac{p}{\varepsilon}}_{L^{\frac{p}{\varepsilon}}(\Omega_T)} + \left\Vert B\right\Vert^p_{L^p(\Omega_T)}\right).
\end{equation*}.
\end{proposition}
\begin{remark}
We stress the fact that Proposition \ref{prop:duality_eps} above requires {\it{a priori}} that the function $u$ lies in $L^p(\Omega_T)$. As a consequence, we shall not be able to directly apply this result to weak solutions of~\eqref{eq:syst_coag-frag}-\eqref{eq:coag}, but only to solutions of an approximate (truncated) system (such as (\ref{qtrunc})), for which we have {\it{a priori}} regularity estimates.
\end{remark}
\begin{proof}
We consider $v$ the solution (whose existence and uniqueness is once again given by Proposition~\ref{prop:exist_para_disc} of the Appendix) of the dual problem 
\begin{equation*}
\left\{
\begin{aligned}
& \partial_t v + M\Delta_x v = -u^{p-1} \quad &\text{on}& \ \Omega_T, \\
& \nabla_x v\cdot \nu = 0 \quad &\text{on}& \ [0,T]\times\partial\Omega, \\
& v(T,\cdot)=0 \quad &\text{on}& \ \Omega.
\end{aligned}
\right.
\end{equation*}
Again multiplying \eqref{eq:duality_eps} by $v$ and integrating on $\Omega_T$, we end up with 
\begin{align}
\label{eq:dual_ineq_2}
\int\limits_{\Omega_T}u^p  &\leq \int\limits_{\Omega}u(0)\,v(0) + \int\limits_{\Omega_T}Au^{1-\varepsilon}v + \int\limits_{\Omega_T}Bv.
\end{align}
Moreover thanks to \eqref{eq:hyp_C3}, we can apply Lemma~\ref{lem:parabolic_estimate} to the above dual problem and get that
\begin{equation*}
\left\Vert v\right\Vert_{L^{p'}(\Omega_T)} \le C \left\Vert u^{p-1}\right\Vert_{L^{p'}(\Omega_T)} \le C \left\Vert u\right\Vert^{p-1}_{L^{p}(\Omega_T)} \quad \text{and} \quad \left\Vert v(0)\right\Vert_{L^{p'}(\Omega)} \le C_0 \left\Vert u\right\Vert^{p-1}_{L^{p}(\Omega_T)}.
\end{equation*} 
Next, returning to \eqref{eq:dual_ineq_2}, we can bound each term of the r.h.s. using several times Young's inequality (sometimes using a parameter $\eta>0$): For the first term, we get
\begin{align*}
\int\limits_{\Omega}u(0)\,v(0) &\leq \frac{1}{p\eta^p}\int\limits_{\Omega}u(0 )^p + \frac{\eta^{p'}}{p'}\int\limits_{\Omega}v(0,\cdot)^{p'}
\leq \frac{1}{p\,\eta^p}\int\limits_{\Omega} u(0)^p + \frac{C_0^{p'}\,\eta^{p'}}{p'}\int\limits_{\Omega_T} u^{p},
\end{align*}
while  for the second one,
\begin{align*}
\int\limits_{\Omega_T}Au^{1-\varepsilon}v  &\leq \frac{1-\varepsilon}{p}\int\limits_{\Omega_T}u^p + \frac{\eta^{p'}}{p'}\int\limits_{\Omega_T}v^{p'} + \frac{\varepsilon}{p\eta^{\frac{p}{\varepsilon}}}\int\limits_{\Omega_T}A^{\frac{p}{\varepsilon}} \\
&\leq \biggl(\frac{1-\varepsilon}{p}+\frac{C^{p'}\,\eta^{p'}}{p'}\biggr)\int\limits_{\Omega_T}u^p + \frac{\varepsilon}{p\,\eta^{\frac{p}{\varepsilon}}}\int\limits_{\Omega_T}A^{\frac{p}{\varepsilon}},
\end{align*}
and finally for the last one,
\begin{align*}
\int\limits_{\Omega_T}Bv  &\leq \frac{1}{p\,\eta^p}\int\limits_{\Omega_T}B^p + \frac{\eta^{p'}}{p'}\int\limits_{\Omega_T}v^{p'} \leq \frac{1}{p\,\eta^p}\int\limits_{\Omega_T}B^p + \frac{C^{p'}\,\eta^{p'}}{p'}\int\limits_{\Omega_T}u^p.
\end{align*}
Putting everything together, we end up with
\begin{equation*}
\int\limits_{\Omega_T}u^p  \leq \left(\frac{1-\varepsilon}{p}+\frac{(2\,C^{p'}+C_0^{p'})\eta^{p'}}{p'}\right)\int\limits_{\Omega_T}u^p + \frac{1}{p\eta^p}\int\limits_{\Omega}u(0,\cdot)^p + \frac{\varepsilon}{p\eta^{\frac{p}{\varepsilon}}}\int\limits_{\Omega_T}A^{\frac{p}{\varepsilon}} + \frac{1}{p\eta^p}\int\limits_{\Omega_T}B^p,
\end{equation*}
and by taking $\eta>0$ small enough, we get the announced estimate.
\end{proof}

\section{Propagation of moments in $L^p$ norms}
\label{sec:prop_Lp}

This Section is devoted to the proof of propagation in $L^p(\Omega_T)$ ($p<+\infty$) 
of moments $\rho_k$.
We begin with 
Proposition~\ref{th:first_moment} and the propagation of the total mass $\rho_1$, when the \textit{closeness} hypothesis \eqref{hyp:closeness} on the diffusion coefficients is satisfied.
\medskip

\noindent \textit{Proof of Proposition~\ref{th:first_moment}.}
For $n\in\N^*$, we consider the solution $c^n=(c_1^n,\ldots,c^n_n)$ of~\eqref{systtrunc}-\eqref{qtrunc}, for which the existence and uniqueness of nonnegative, global and smooth solutions is classical (see for example Proposition 2.1 and Lemma 2.2 of \cite{Wrz97}, or \cite{DMilan}). Summing up the equations \eqref{systtrunc} for each $i$, we get
\begin{equation*}
\partial_t \left(\sum_{i=1}^{n} i c_i^n\right) - \Delta_x\left(\sum_{i=1}^{n} id_i c_i^n\right) =0,
\end{equation*}
which rewrites, when 
\begin{equation*}
\rho_1^n=\sum_{i=1}^n ic_i^n \quad \text{and}\quad  M_1^n:=\frac{\sum_{i=1}^{n} i d_ic_i^n}{\sum_{i=1}^{\infty} i c_i^n},
\end{equation*} as
\begin{equation*}
\partial_t \rho_1^n - \Delta_x\left(M_1^n\,\rho_1^n\right) =0.
\end{equation*}
Using
\begin{equation*}
a=\inf\limits_{i\geq 1} \{d_i\}  \quad \text{and} \quad b=\sup\limits_{i\geq 1} \{d_i\},
\end{equation*}
we get $a\leq M_1^n\leq b$ independently of $n$. Proposition~\ref{prop:duality} then yields
\begin{equation*}
\left\Vert\rho_1^n\right\Vert_{L^p(\Omega_T)} \leq C \left\Vert\rho_1^n(0,\cdot)\right\Vert_{L^p(\Omega)} \leq C \left\Vert\rho_1^{in}\right\Vert_{L^p(\Omega)},
\end{equation*}
where $C$ does not depend on $n$. By Proposition~\ref{prop:extraction} (or respectively \cite[Theorem 3]{LauMis02}),  we get a weak solution $c=(c_i)_{i\in\N*}$ of~\eqref{eq:syst_coag-frag}-\eqref{eq:coag} defined by (up to extraction)
\begin{equation*}
c_i=\lim\limits_{n\to\infty} c_i^n,
\end{equation*}
and thanks to Fatou's Lemma, we see that $\left\Vert\rho_1\right\Vert_{L^p(\Omega_T)} \leq C \left\Vert\rho_1^0\right\Vert_{L^p(\Omega)}.$ \hfill $\qed$

\begin{remark}
In fact Proposition~\ref{th:first_moment} would be valid for any weak solution to~\eqref{eq:syst_coag-frag}-\eqref{eq:coag} such that $\sum_{i=1}^{\infty}iQ_i(c) \in L^1(\Omega_T)$. Indeed, one can then prove that
\begin{equation}\label{eq:weak_mass_conserv}
\partial_t \left(\sum_{i=1}^{\infty} i c_i\right) - \Delta_x\left(\sum_{i=1}^{\infty} id_i c_i\right) =0
\end{equation}
holds weakly, and one can then directly apply Proposition~\ref{prop:duality} to~\eqref{eq:weak_mass_conserv}.
\end{remark}

\medskip

The proof of Proposition~\ref{th:all_moments} is a bit more involved but still based on the same idea.
 The outline of the proof is the following: First, we get
 $L^{\infty}(\Omega_T)$ bounds for each concentration $c_i$ and for any finite time $T$. Thus, it is sufficient to prove propagation in $L^p$ spaces for tail moments, in
 which we only consider concentrations $c_i$ for $i$ larger than some index $I$. Because we assumed that the $d_i$ 
 converge (when $i \to \infty$) towards a strictly positive real number, 
 the closeness hypothesis~\eqref{hyp:closeness} will always be satisfied for the coefficients $\left(d_i\right)_{i\geq I}$ when $I$ is large enough.
 This allows us to use a similar argument as in Proposition~\ref{th:first_moment} to prove the propagation in $L^p(\Omega_T)$ of the mass and then of all higher order moments.

\bigskip 

\noindent \textit{Proof of Proposition~\ref{th:all_moments}.}
As for Proposition~\ref{th:first_moment}, the rigorous way to prove Proposition~\ref{th:all_moments} is to get all the needed estimates on the solutions of the truncated problems~\eqref{systtrunc}-\eqref{qtrunc} and then pass to the limit (when $n \to \infty$). However for a clearer exposition of the different arguments, we first derive (sometimes formally) estimates on the whole system~\eqref{eq:syst_coag-frag}-\eqref{eq:coag} and then explain how to pass to the limit in the corresponding estimates on the truncated system. We begin with 
\begin{lemma}\label{lem:c_i}
Let $\Omega$ be a smooth bounded domain of $\R^N$. Assume that the coagulation coefficients satisfy~\eqref{hyp:exist_LM} and that $d_i>0$ for all $i\in \N^*$.
Assume also that each $c_i^{in}\ge 0$ lies in $L^{\infty}(\Omega)$. 
 We consider a global weak nonnegative solution of \eqref{eq:syst_coag-frag}-\eqref{eq:coag} (nonnegative meaning here that $c_i\ge 0$ for all $i \in \N^*$).
 
Then, the concentration $c_i$ 
 lies in $L^{\infty}(\Omega_T)$ for each integer $i\in\N^*$ and any positive time $T>0 $.
\end{lemma}
\begin{proof}
 Since
\begin{equation*}
\partial_t c_1 - d_1 \Delta_x c_1 \leq Q_1^+(c) = 0,
\end{equation*}
the maximum principle for the heat equation yields
 that $c_1\in L^{\infty}(\Omega_T)$. Then, we observe that for all $i\geq 2$,
\begin{equation*}
\partial_t c_i - d_i \Delta_x c_i \leq  Q_i^+(c).
\end{equation*}
Since the coagulation gain term $\ds Q_i^+(c)=\frac{1}{2}\sum_{j=1}^{i-1}a_{i,j}c_{i-j}c_j$ involves only $c_j$ for $j<i$,
 we can conclude the statement of the lemma by induction.
\end{proof}

The proof of Lemma \ref{lem:c_i} shows sufficient conditions under which each $c_i$ is bounded on $\Omega_T$,
 but explicit bounds computed in this way would grow very fast with $i$. 
Thus, there is little hope of obtaining a result on $\rho_1$ by directly using this method. However, the knowledge that any finite truncation of $\rho_1$ lies in $L^{\infty}(\Omega_T)$ enables us to prove another result of propagation of $L^p$ norms for the mass $\rho_1$, where the assumption \eqref{hyp:closeness} is removed and replaced by the assumption of convergence of the diffusion coefficients $d_i$ towards a strictly positive limit.

\begin{lemma}
\label{th:first_moment_convergence}
Let $\Omega$ be a smooth bounded domain of $\R^N$. Assume that the coagulation coefficients satisfy~\eqref{sub}. Assume also that all $d_i$ are strictly positive, 
and that $(d_i)$ converges toward a strictly positive limit. Finally, assume that each $c_i^{in}\ge 0$ lies in $L^{\infty}(\Omega)$ and that $\rho_1^{in}\in L^p(\Omega)$ for some $p\in]1,+\infty[$.
We consider a global weak nonnegative solution of \eqref{eq:syst_coag-frag}-\eqref{eq:coag} (nonnegative meaning here that $c_i\ge 0$ for all $i \in \N^*$).

Then, 
 $\rho_1\in L^p(\Omega_T)$, for any finite time $T>0$.
\end{lemma}
\begin{proof}
We define
\begin{equation*}
a^I:=\inf\limits_{i\geq I} d_i, \qquad \text{and} \qquad b^I:=\sup\limits_{i\geq I} d_i.
\end{equation*}
Since $(d_i)$ converges toward a positive limit, there exist a positive integer $I$ for all $p'\in]1,+\infty[$ such that
\begin{equation}
\label{eq:hyp_I}
\frac{b^I-a^I}{b^I+a^I}\, \constC_{\frac{a^I+b^I}{2},p'}<1.
\end{equation}
We then consider 
\begin{equation*}
\rho_1^I:=\sum_{i=I}^{\infty} ic_i, \qquad\text{and}\qquad  M^I_1:=\frac{\sum_{i=I}^{\infty} i d_ic_i}{\sum_{i=I}^{\infty} i c_i}.
\end{equation*}
Note that thanks to Lemma~\ref{lem:c_i}, it is enough to prove that $\rho_1^I\in L^p(\Omega_T)$ in order
to conclude the proof of Lemma \ref{th:first_moment_convergence}.
We therefore compute (remember that $a_{i,j}= a_{j,i}$)
\begin{align*}
\partial_t \rho_1^I - \Delta_x \left(M_1^I\rho_1^I\right) &= \frac{1}{2} \sum_{i=1}^{\infty}\sum_{j=1}^{\infty}a_{i,j}c_ic_j\left((i+j)\mathds{1}_{i+j\geq I}-i\mathds{1}_{i\geq I}-j\mathds{1}_{j\geq I}\right)\\
&= \frac{1}{2} \sum_{i=1}^{\infty}\sum_{j=1}^{\infty}a_{i,j}c_ic_j\left(i\left(\mathds{1}_{i+j\geq I}-\mathds{1}_{i\geq I}\right)+j\left(\mathds{1}_{i+j\geq I}-\mathds{1}_{j\geq I}\right)\right)\\
&= \sum_{i=1}^{\infty}\sum_{j=1}^{\infty}a_{i,j}c_ic_ji\left(\mathds{1}_{i+j\geq I}-\mathds{1}_{i\geq I}\right).
\end{align*}
Next, by using assumption \eqref{sub}, and more precisely that $a_{i,j}\leq C\, (i^{\gamma}+j^{\gamma})\leq C\, (i+j)$, we have 
\begin{align*}
\sum_{i=1}^{\infty}\sum_{j=1}^{\infty}a_{i,j}c_ic_ji\left(\mathds{1}_{i+j\geq I}-\mathds{1}_{i\geq I}\right) 
&\le C\, \sum_{i=1}^{\infty}\sum_{j=1}^{\infty} (i+j)\,c_ic_j i\left(\mathds{1}_{i+j\geq I}-\mathds{1}_{i\geq I}\right)  \\
&\leq C\,\sum_{i=1}^{I-1}\sum_{j=1}^{\infty} i^{2} c_ic_j  + 
C\,\sum_{i=1}^{I-1}\sum_{j=1}^{\infty} ic_ijc_j  \\
&\le 2C \bigg( \sum_{i=1}^{I-1}i^2\, c_i \bigg)\, \bigg(\sum_{j=1}^{\infty}j\,c_j \bigg).
\end{align*}
Thus, we obtain 
\begin{align*}
\partial_t \rho_1^I - \Delta_x \left(M_1^I\rho_1^I\right)
&\leq \psi_1\rho_1^I + \psi_2,
\end{align*}
where $\ds \psi_1= 2C\sum_{i=1}^{I-1}i^2 \,c_i$, and $\ds \psi_2= \psi_1\sum_{i=1}^{I-1} j\,c_j $. Now thanks to Lemma~\ref{lem:c_i}, both $\psi_1$ and $\psi_2$ belong to $L^{\infty}(\Omega_T)$. Then, if we denote (for $i\in\{1,2\}$), $\mu_i := \left\Vert \psi_i \right\Vert_{L^{\infty}(\Omega_T)}$, we get
\begin{equation*}
\partial_t \rho_1^I - \Delta_x \left(M_1^I\rho_1^I\right) \leq  \mu_1 \rho_1^I + \mu_2,
\end{equation*}
and we can conclude using Proposition~\ref{prop:duality_rhs}.
\end{proof}

\begin{remark}
Note that aside from symmetry, the above proof only requires the estimate $a_{i,j}\leq C\,i\,j$, which is a much weaker restriction on the coagulation coefficients than the ``sublinear'' assumption~\eqref{sub}. However, ``strictly superlinear'' coagulation is known to produce gelation already in the spatially homogeneous case.
Nonetheless, it is known for the homogeneous case that adding sufficiently strong fragmentation in the model can prevent gelation even with ``superlinear'' coagulation. Similar results in presence of diffusion, together with generalisations of some results of this paper to  models including fragmentation will be discussed in \cite{Bre15}.
\end{remark}
\medskip

\noindent \textit{Continuation of the proof of Proposition~\ref{th:all_moments}.}
We shall now prove the propagation of $L^p$, ($p<+\infty$) regularity for moments of higher order. This is done still under the assumption that the diffusion coefficients $d_i$ converge towards a
strictly positive limit. 
\medskip

We first introduce
\begin{equation*}
M^I_k:=\frac{\sum_{i=I}^{\infty} i^k d_ic_i}{\sum_{i=I}^{\infty} i^k c_i}, \qquad a^I \le M^I_k \le b^I.
\end{equation*}

The proof of propagation for moments of higher order involves the {\sl{a priori}} estimate established in Proposition~\ref{prop:duality_eps}. Therefore the results of Proposition~\ref{th:all_moments} (for $k>1$) only apply to such solutions, which are constructed as a limit of solutions of a truncated system. 
For a clearer exposition of the proof, we first perform the computations formally and then show afterwards how to conclude rigorously through the use of the truncated systems~\eqref{systtrunc}-\eqref{qtrunc}.
\smallskip

 We proceed by induction and assume that (for some integer $k$), $\rho_l\in L^p(\Omega_T)$, for all $p<+\infty$ and every $l\leq k-1$. Note that 
 Lemma \ref{th:first_moment_convergence} ensures 
that the induction hypothesis holds for $k=2$.
For any $I\in\N^*$ (using \eqref{sub}), we have
\begin{align}
\label{eq:moments}
\partial_t \rho_{k}^I -\Delta_x(M^I_{k}\rho^I_{k}) &\leq \frac{1}{2} \sum_{i=1}^{\infty}\sum_{j=1}^{\infty}a_{i,j}c_ic_j\left((i+j)^{k}\mathds{1}_{i+j\geq I}-i^{k}\mathds{1}_{i\geq I}-j^{k}\mathds{1}_{j\geq I}\right) \nonumber\\
&\leq \frac{C}{2} \sum_{l=1}^{{k}-1}\binom{k}{l} \sum_{i=1}^{\infty}\sum_{j=1}^{\infty} (i^{\gamma}+j^{\gamma})\,i^lc_i\,j^{{k}-l}c_j \nonumber\\ 
&\quad +\frac{C}{2} \sum_{i=1}^{\infty}\sum_{j=1}^{\infty}(i^{\gamma}+j^{\gamma})c_ic_j(i^{k}(\mathds{1}_{i+j\geq I}-\mathds{1}_{i\geq I})+j^{k}(\mathds{1}_{i+j\geq I}-\mathds{1}_{j\geq I})) \nonumber\\
&\leq C \sum_{l=1}^{{k}-1}\binom{{k}}{l} \sum_{i=1}^{\infty}i^{\gamma+l}c_i\sum_{j=1}^{\infty} j^{k-l}c_j \nonumber\\ 
&\quad + C \sum_{i=1}^{I-1}i^{\gamma+k}c_i\sum_{j=1}^{\infty} c_j + C\sum_{i=1}^{I-1}i^{k}c_i\sum_{j=1}^{\infty}j^{\gamma}c_j . 
\end{align}
In the special case $k=2$, this yields (using $\gamma\leq 1$)
\begin{equation*}
\partial_t \rho_{2}^I -\Delta_x(M^I_{2}\rho^I_{2}) \leq 2C\rho_{1+\gamma}\rho_1 +2C\rho_1 \sum_{i=1}^{I-1}i^{3}c_i.
\end{equation*}
Moreover, we can split the moment of order $1+\gamma$ between a finite part that we already control and a tail, and then bound the latter using H\"older's inequality:
\begin{equation*}
\rho_{1+\gamma} \leq \sum_{i=1}^{I-1}i^{1+\gamma}c_i + \rho_{1+\gamma}^I \leq \sum_{i=1}^{I-1}i^{2}c_i +  \left(\rho_{2}^I\right)^{1-\varepsilon} \left(\rho_1^I\right)^{\varepsilon},\qquad
\text{where}\quad\varepsilon=1-\gamma>0.
\end{equation*}
Therefore, we end up with
\begin{equation}
\label{eq:moment_2}
\partial_t \rho_{2}^I -\Delta_x(M^I_{2}\rho^I_{2}) \leq 2C\left(\rho_{2}^I\right)^{1-\varepsilon} \left(\rho_1\right)^{1+\varepsilon} + 4C\rho_1 \sum_{i=1}^{I-1}i^{3}c_i,
\end{equation}
and the last term in the r.h.s. of \eqref{eq:moment_2} lies in $L^p(\Omega_T)$ for every $p<+\infty$ thanks to Lemma~\ref{lem:c_i} and Lemma~\ref{th:first_moment_convergence}. Thus, taking $I$ large enough, inequality \eqref{eq:moment_2} and Proposition~\ref{prop:duality_eps} formally yield a (computable) bound for $\rho_2^I$ (and therefore $\rho_2$) in $L^p(\Omega_T)$, for all $p<+\infty$. 
 
Next, returning to the general case of moments of order $k>2$, we 
 estimate \eqref{eq:moments} as
\begin{align*}
\partial_t \rho_{k}^I -\Delta_x(M^I_{k}\rho^I_{k}) &\leq kC \sum_{i=1}^{\infty}i^{\gamma+k-1}c_i\sum_{j=1}^{\infty} jc_j + C \sum_{l=1}^{{k}-2}\binom{{k}}{l}  \sum_{i=1}^{\infty}i^{\gamma+l}c_i\sum_{j=1}^{\infty} j^{k-l}c_j \nonumber\\ 
&\quad + C \sum_{i=1}^{I-1}i^{\gamma+k}c_i\sum_{j=1}^{\infty} c_j + C\sum_{i=1}^{I-1}i^{k}c_i\sum_{j=1}^{\infty}j^{\gamma}c_j  \nonumber \\
&\leq kC\rho_{k-1+\gamma}\rho_1 + C \sum_{l=1}^{{k}-2}\binom{{k}}{l} \rho_{l+1}\rho_{k-l} +2C\rho_1 \sum_{i=1}^{I-1}i^{k+1}c_i,
\end{align*}
where we used $\gamma\leq 1$. Again, we can split the moment of order $k-1+\gamma$ between a finite part that we already control and a tail, and then bound the latter using H\"older's inequality:
\begin{equation*}
\rho_{k-1+\gamma} \leq \sum_{i=1}^{I-1}i^{k-1+\gamma}c_i + \rho_{k-1+\gamma}^I \leq \sum_{i=1}^{I-1}i^{k}c_i +  \left(\rho_{k}^I\right)^{1-\varepsilon} \left(\rho_1^I\right)^{\varepsilon},\qquad
\text{where}\quad\varepsilon=\frac{1-\gamma}{k-1}>0.
\end{equation*}
Therefore, we end up with
\begin{equation}
\label{eq:moment_k2}
\partial_t \rho_{k}^I -\Delta_x(M^I_{k}\rho^I_{k}) \leq kC\left(\rho_{k}^I\right)^{1-\varepsilon} \left(\rho_1\right)^{1+\varepsilon} + C \sum_{l=1}^{{k}-2}\binom{{k}}{l} \rho_{l+1}\rho_{k-l} +(k+2)C\rho_1 \sum_{i=1}^{I-1}i^{k+1}c_i,
\end{equation}
and the last two terms in the r.h.s. of \eqref{eq:moment_k2} lie in $L^p(\Omega_T)$ for every $p<+\infty$ thanks to Lemma~\ref{lem:c_i}, Lemma~\ref{th:first_moment_convergence} and the induction hypothesis. Thus, taking $I$ large enough, inequality \eqref{eq:moment_k2} and Proposition~\ref{prop:duality_eps} would yield a (computable) bound for $\rho_k^I$ (and therefore $\rho_k$) in $L^p(\Omega_T)$, for all $p<+\infty$, except that Proposition~\ref{prop:duality_eps} 
also requires to {\sl{a priori}} know that $\rho_k^I \in L^p(\Omega_T)$, which is not the case at this point.
\medskip

 In order to make the proof rigorous, we need to apply Proposition~\ref{prop:duality_eps} to smooth solutions obtained by truncating the original system. Therefore, for an integer $n>I$, we consider $c^n=\left(c_i^n,\ldots,c_n^n\right)$ the solution of~\eqref{systtrunc}-\eqref{qtrunc} and define
\begin{equation*}
\rho_k^n = \sum_{i=1}^n i^kc_i^n, \quad \rho_k^{n,I} = \sum_{i=I}^{n}i^kc_i^n \quad \text{and}\quad M_k^{n,I}=\frac{\sum_{i=I}^{n} i^k d_ic_i^n}{\sum_{i=I}^{n} i^k c_i^n}.
\end{equation*}
We then perform the same computations as previously, taking into account the truncation
in the coagulation kernel (\ref{qtrunc}).
We get for the second order moment (and $n>I$)
\begin{align*}
\partial_t \rho_{2}^{n,I} -\Delta_x(M^{n,I}_{2}\rho^{n,I}_{2}) & \leq 2C\left(\rho_{2}^{n,I}\right)^{1-\varepsilon} \left(\rho_1^n\right)^{1+\varepsilon} + 4C\rho_1^n \sum_{i=1}^{I-1}i^{3}c_i^n\nonumber\\
 & = A_1^n\left(\rho^{n,I}_2\right)^{1-\varepsilon} + B_1^{n,I},
\end{align*}
where $\varepsilon=1-\gamma>0$ and $A_1^{n}$ and $B_1^{n,I}$ only depend on the approximating first order moment $\rho_1^{n}$ and on a finite number of approximate concentrations $c_i^n$, for $i<I$. Since this time we know that $\rho^{n,I}_2\in L^p(\Omega_T)$, we can apply Proposition~\ref{prop:duality_eps} and get the estimate  
\begin{align}
\label{eq:moment_2approx}
\int_{\Omega_T} \left(\rho^{n,I}_2\right)^p &\leq C \left(\int_{\Omega} \left(\rho^{n,I}_2(0)\right)^p + \int_{\Omega_T} \left(A_1^{n}\right)^\frac{p}{\varepsilon} + \int_{\Omega_T} \left(B_1^{n,I}\right)^p\right) \nonumber\\
&\leq C \left( \int_{\Omega} \left(\rho_2^{in}\right)^p + \int_{\Omega_T} \left(A_1^{n}\right)^\frac{p}{\varepsilon} + \int_{\Omega_T} \left(B_1^{n,I}\right)^p \right),
\end{align} 
where the constant $C$ does not depend on $n$. In order to complete the proof, we still
have to show that $A_1^{n}$ and $B_1^{n,I}$ can be bounded in $L^p$ norms uniformly-in-$n$.

Prior to that, we consider also the approximation of any general moments of order $k>2$. 
By estimating as in the computation leading to
\eqref{eq:moment_k2}, we obtain for the truncated moments of order $k>2$
\begin{align*}
\partial_t \rho_k^{n,I} - \Delta_x \left(M^{n,I}_k\rho_k^{n,I}\right) &\leq 
kC\left(\rho_{k}^{n,I}\right)^{1-\varepsilon} \left(\rho_1^n\right)^{1+\varepsilon} \nonumber\\
&\quad + C \sum_{l=1}^{{k}-2}\binom{{k}}{l} \rho_{l+1}^n\rho_{k-l}^n +(k+2)C\rho_1^n \sum_{i=1}^{I-1}i^{k+1}c_i^n \nonumber\\
&= A_{k-1}^{n}\left(\rho^{n,I}_k\right)^{1-\varepsilon} + B_{k-1}^{n,I},
\end{align*}
with $\varepsilon=\frac{1-\gamma}{k-1}>0$, and where $A_{k-1}^{n}$ and $B_{k-1}^{n,I}$ only depend on moments $\rho_l^{n}$ of integer order $l$ between $1$ and $k-1$ and on a finite number of concentrations $c_i^n$ (for $i<I$). Moreover, since $\rho^{n,I}_k\in L^p(\Omega_T)$, we can again apply Proposition~\ref{prop:duality_eps} and estimate 
\begin{align}
\label{eq:moment_kapprox}
\int_{\Omega_T} \left(\rho^{n,I}_k\right)^p 
&\leq C \left( \int_{\Omega} \left(\rho_k^{in}\right)^p + \int_{\Omega_T} \left(A_{k-1}^n\right)^\frac{p}{\varepsilon} + \int_{\Omega_T} \left(B_{k-1}^{n,I}\right)^p \right),
\end{align} 
where the constant $C$ does not depend on $n$.
 So we again have to prove  that $A_{k-1}^n$ and $B_{k-1}^{n,I}$ can be bounded in $L^p$ 
norms uniformly-in-$n$.

First we notice that for any given $i$, since $c_i^{in}\in L^{\infty}(\Omega)$, the concentrations $c_i^{n}$ can be bounded in $L^{\infty}(\Omega_T)$ uniformly-in-$n$ by the computations of Lemma~\ref{lem:c_i}. Indeed,
\begin{equation*}
\partial_t c_1^n - d_1 \Delta_x c_1^n \leq Q_1^{+,n}(c^n) = 0
\end{equation*}
yields a uniform-in-$n$ bound for $c_1^n$, and then
\begin{equation*}
\partial_t c_i^n - d_i \Delta_x c_i^n \leq  \frac{1}{2}\sum_{j=1}^{i-1}a_{i-j,j}c_{i-j}^nc_j^n
\end{equation*} 
allows to conclude inductively in $i$ for all $T>0$. 
Now, from that fact (that is, each $c_i^n$ is uniformly-in-$n$ bounded in $L^{\infty}(\Omega_T)$), we get that $\rho_1^{n,I}$ (and also $\rho_1^{n}$) is uniformly-in-$n$ bounded in any $L^p$ norm, $p<+\infty$, thanks to Lemma~\ref{th:first_moment_convergence} and Proposition~\ref{prop:duality_rhs} (one just needs to repeat the computations in the proof of Lemma~\ref{th:first_moment_convergence} with $\rho_1^{n,I}$ instead of $\rho_1^I$). Therefore $A_1^n$ and $B_1^{n,I}$ are uniformly-in-$n$ bounded in any $L^p(\Omega_T)$, $p<+\infty$, and going back to~\eqref{eq:moment_2approx}, this yields that $\rho^{n,I}_2$ (and thus $\rho_2^n$) is uniformly-in-$n$ bounded in any $L^p(\Omega_T)$, $p<+\infty$. Similarly, we can prove inductively using \eqref{eq:moment_kapprox} that for all $k>2$, $\rho^{n,I}_k$ (and thus $\rho_k^n$) is uniformly-in-$n$ bounded in any $L^p(\Omega_T)$, $p<+\infty$. Applying Fatou's Lemma we can conclude that, for the weak solutions given by Proposition~\ref{prop:extraction}, $\rho_k$ is bounded in any $L^p(\Omega_T)$, $p<+\infty$. \hfill $\qed$

\begin{remark}
Fatou's Lemma is enough here to show that $\rho_k\in L^p(\Omega_T)$, but since for any $k\in\N^*$ we know that $\rho_{k+1}^n$ is bounded in $L^p(\Omega_T)$ uniformly-in-$n$, we could show by interpolation that we do in fact have the convergence of $\rho_k^n$ to $\rho_k$ in $L^p(\Omega_T)$.
\end{remark}

\begin{remark}
Note that Proposition~\ref{th:all_moments} states that we have propagation of the moment $\rho_k$ in every $L^p(\Omega_T)$, $p<+\infty$, provided that the initial moment $\rho_k^{in}$ lies in every $L^p(\Omega)$, $p<+\infty$. If we only want to get propagation of $\rho_k$ in $L^p(\Omega_T)$ for some fixed $p$, we can relax a bit the hypothesis, but to apply the above proof we still need to assume that initial moment of lower order $\rho_l^{in}$, $l<k$, are in some space $L^q(\Omega)$ with $q>p$ depending of the magnitude of the coagulation. For instance, if we want to get for some fixed $p$ that $\rho_2\in L^p(\Omega_T)$ with the method of Proposition~\ref{th:all_moments}, we need to assume that $\rho_2^{in}\in L^p(\Omega)$ and $\rho_1^{in}\in L^q(\Omega)$, where $q=\frac{2-\gamma}{1-\gamma}p$.
\end{remark}

\begin{remark}
\label{remark:d_i=0}
Finally, we point out that Proposition~\ref{th:all_moments} would still hold if a finite number of diffusion coefficients $d_i$ were equal to $0$. Indeed Lemma~\ref{lem:c_i} is still valid in this case: if $c_j\in L^{\infty}(\Omega_T)$ for all $j<i$ and $d_i=0$, then
\begin{equation*}
\partial_t c_i \leq  Q_i^+(c)=\frac{1}{2}\sum_{j=1}^{i-1}a_{i,j}c_{i-j}c_j
\end{equation*}
shows that $c_i\in L^{\infty}(\Omega_T)$; and up to taking $I$ large enough, we would still get $\inf\limits_{i\geq I} d_i>0$, so we could still apply Proposition~\ref{prop:duality_rhs} and Proposition~\ref{prop:duality_eps} to control the tail moments $\rho_k^I$.
\end{remark}

\section{Propagation of Sobolev norms for moments}
\label{sec:max_reg}

In this Section, we show how the results of Proposition~\ref{th:all_moments} can be improved
in order to get higher regularity as stated in Theorem~\ref{th:all_moments_w}. 
We also explain how the obtained regularity in fact implies uniqueness.

\medskip 
\noindent\textit{Proof of Theorem~\ref{th:all_moments_w}.}
We consider a solution provided by Proposition~\ref{th:all_moments}, for which we already know
that we have propagation of moments in $L^p$ spaces. Remembering~\eqref{eq:syst_coag-frag}, we want to use the properties of the heat equation to get additional regularity, and to do so we first need to estimate the coagulation term. This is the content of the following lemma. 
\begin{lemma}\label{lem:ci_to_Qi}
Let $\Omega$ be a smooth bounded domain of $\R^N$ and $s\in\N$. Assume that $(c_i)_{i\in\N^*}$ is a sequence of positive functions defined on $\Omega_T$ such that
\begin{equation}
\label{eq:ikc_i}
\sup\limits_{i\geq 1} \left\Vert i^k c_i \right\Vert_{W^{s,p}(\Omega_T)} < +\infty,\quad \forall k\in\N,\ \forall p<+\infty.
\end{equation}
Then, assuming~\eqref{eq:coag} and \eqref{sub}, the following estimates hold:
\begin{equation}
\label{eq:ikQ_i}
\sup\limits_{i\geq 1} \left\Vert i^k Q_i(c) \right\Vert_{W^{s,p}(\Omega_T)} < +\infty,\quad \forall k\in\N,\ \forall p<+\infty.
\end{equation}
\end{lemma}
\begin{proof}
Remembering~\eqref{eq:coag} and using the sublinearity of the coagulation coefficients~\eqref{sub} ($a_{i,j}\leq C\, (i^{\gamma}+j^{\gamma})\leq C\, (i+j)\le 2C\,ij$), we estimate
\begin{align*}
\left\Vert i^kQ_i(c) \right\Vert_{W^{s,p}(\Omega_T)} \leq 
C\sum_{j=1}^{i-1} i^{k+1}\left\Vert c_{i-j} c_j \right\Vert_{W^{s,p}(\Omega_T)} + 2C\big\Vert i^{k+1} c_i\sum_{j=1}^{\infty} jc_j \big\Vert_{W^{s,p}(\Omega_T)}.
\end{align*}
We now use the algebra property of $\ds\bigcap_{p<+\infty} W^{s,p}(\Omega_T)$. More precisely, by combining Cauchy-Schwarz's inequality and Leibnitz's formula, the following estimate holds:
\begin{equation*}
\left\Vert uv \right\Vert_{W^{s,p}(\Omega_T)} \leq  C(s)\left\Vert u \right\Vert_{W^{s,2p}(\Omega_T)} \left\Vert v \right\Vert_{W^{s,2p}(\Omega_T)},
\end{equation*}
where $C(s)$ is a constant depending only on $s$. Therefore,
\begin{align*}
\left\Vert i^kQ_i(c) \right\Vert_{W^{s,p}(\Omega_T)} \leq 
C(s) \Biggl(\,&\sum_{j=1}^{i-1} i^{k+1}\left\Vert c_{i-j} \right\Vert_{W^{s,2p}(\Omega_T)} \left\Vert c_{j} \right\Vert_{W^{s,2p}(\Omega_T)} \\
& + \left\Vert i^{k+1}c_i\right\Vert_{W^{s,2p}(\Omega_T)} \sum_{j=1}^{\infty}\left\Vert j c_j \right\Vert_{W^{s,2p}(\Omega_T)} \biggr) \\
=  C(s) \Biggl(\,&\sum_{j=1}^{i-1} \frac{i^{k+1}}{(i-j)^{k+2}j^{k+2}}\left\Vert (i-j)^{k+2}c_{i-j} \right\Vert_{W^{s,2p}(\Omega_T)} \left\Vert j^{k+2}c_{j} \right\Vert_{W^{s,2p}(\Omega_T)} \\
&  + \left\Vert i^{k+1}c_i\right\Vert_{W^{s,2p}(\Omega_T)} \sum_{j=1}^{\infty}\frac{1}{j^2} \left\Vert j^{3} c_j \right\Vert_{W^{s,2p}(\Omega_T)}\Biggr).
\end{align*}
Using~\eqref{eq:ikc_i}, we get
\begin{equation*}
\left\Vert i^kQ_i(c) \right\Vert_{W^{s,p}(\Omega_T)} \leq C(s,p,k) \left(1+\sum_{j=1}^{i-1} \frac{i^{k+1}}{(i-j)^{k+2}j^{k+2}}\right),
\end{equation*}
where $C(s,p,k)$ depends on the quantities $\sup\limits_{j\geq 1} \left\Vert j^l c_j \right\Vert_{W^{s,2p}(\Omega_T)}$ for $l\in \N$,
but not on $i$. We then show that (for any $k\in\N$) 
\begin{equation*}
\sup_{i\ge 1}\, \ds \sum_{j=1}^{i-1}\frac{i^{k+1}}{(i-j)^{k+2}j^{k+2}} < \infty .
\end{equation*}
Indeed, by symmetry, we know that (denoting by $[m]$ the integer part of $m$)
\begin{align*}
 \sup_{i\ge 1}\, \sum_{j=1}^{i-1}\frac{i^{k+1}}{j^{k+2}(i-j)^{k+2}} &
 \leq 2 \sup_{i \ge 1} \sum_{j=1}^{\left[\frac{i}{2}\right]}\frac{i^{k+1}}{j^{k+2}(i-j)^{k+2}} 
 \leq 2 \sup_{i\ge 1} \sum_{j=1}^{\left[\frac{i}{2}\right]}\left(\frac{i}{i-\left[\frac{i}{2}\right]}\right)^{k+2}\frac{1}{j^{k+2}} \\
& \leq 2^{k+3} \sum_{j=1}^{\infty}\frac{1}{j^{k+2}} < \infty .
\end{align*}
This implies that
\begin{equation*}
\sup\limits_{i\geq 1} \left\Vert i^k Q_i(c) \right\Vert_{W^{s,p}(\Omega_T)} < +\infty,\quad \forall k\in\N,\ \forall p<+\infty,
\end{equation*}
and Lemma \ref{lem:ci_to_Qi} is proven.
\end{proof}
\par

\noindent \textit{Continuation of the proof of Theorem~\ref{th:all_moments_w}.}
Now, we can show that under the hypothesis of Theorem~\ref{th:all_moments_w}, the concentrations $(c_i)$ considered in Proposition~\ref{th:all_moments} satisfy
\begin{equation}
\label{eq:ikc_i_s}
\sup\limits_{i\geq 1} \left\Vert i^k c_i \right\Vert_{W^{s,p}(\Omega_T)} < +\infty,\quad \forall k\in\N,\ \forall p<+\infty,\ \forall s\in\N.
\end{equation}
We shall prove \eqref{eq:ikc_i_s} by induction on $s$. The case $s=0$ is a direct consequence of Proposition~\ref{th:all_moments}. Then, if for some $s\in\N$,
\begin{equation*}
\sup\limits_{i\geq 1} \left\Vert i^k c_i \right\Vert_{W^{s,p}(\Omega_T)} < +\infty,\quad \forall k\in\N,\ \forall p<+\infty,
\end{equation*}
we see that Lemma~\ref{lem:ci_to_Qi} yields the estimate
\begin{equation*}
\sup\limits_{i\geq 1} \left\Vert i^k Q_i(c) \right\Vert_{W^{s,p}(\Omega_T)} < +\infty,\quad \forall k\in\N,\ \forall p<+\infty.
\end{equation*}
Remembering that $i^k\,c_i$ satisfies
\begin{equation*}
\left\{
\begin{aligned}
& \left(\partial_t - d_i \Delta_x\right) i^kc_i = i^kQ_i(c) \quad &\text{on}& \ \Omega_T ,\\
& \nabla_x (i^k c_i)\cdot \nu = 0 \quad &\text{on}& \ [0,T]\times\partial\Omega, \\
& i^kc_i(0,\cdot) = i^kc_i^{in} \quad &\text{on}& \ \Omega,
\end{aligned}
\right.
\end{equation*}
and using the regularising properties of the heat equation (they can be used uniformly w.r.t. $i$ since the
diffusion rates $d_i$ are bounded above and below by strictly positive constants), we get the estimate
\begin{equation*}
\sup\limits_{i\geq 1} \left\Vert i^k c_i \right\Vert_{W^{s+1,p}(\Omega_T)} < +\infty,
\quad \forall k\in\N,\ \forall p<+\infty.
\end{equation*}
This concludes the proof of~\eqref{eq:ikc_i_s}. Notice that we also get $W^{s,p}$ estimates for polynomial moments of any order, because
\begin{equation*}
\left\Vert \rho_k \right\Vert_{W^{s,p}(\Omega_T)} \leq \sum_{i=1}^{\infty} \frac{1}{i^2} \left\Vert i^{k+2}c_i\right\Vert_{W^{s,p}(\Omega_T)} \leq \sup\limits_{i\geq 1} \left\Vert i^{k+2}c_i\right\Vert_{W^{s,p}(\Omega_T)}\sum_{i=1}^{\infty} \frac{1}{i^2}.
\end{equation*} 
The $\C^{\infty}$ regularity announced in Theorem~\ref{th:all_moments_w} is now a straightforward consequence of Sobolev embeddings (note that while $\Omega$ is assumed to be smooth, $\Omega_T=\Omega\times]0,T[$ will be only of Lipschitz regularity,
 but this is enough to apply the required Sobolev embeddings (see for instance \cite{Ada03})). The uniqueness of such a smooth solution is given by a
 straightforward extension of \cite[Theorem~1.4]{HamRez07} to the case
 of bounded smooth domains with Neumann boundary conditions (the uniqueness Theorem of \cite{HamRez07} is stated when the spatial domain is $\R^N$), where it is proven that there cannot exist more than one weak solution to~\eqref{eq:syst_coag-frag}-\eqref{eq:coag} satisfying $\rho_2\in L^{\infty}(\Omega_T)$, as soon as the coagulation coefficients satisfy $a_{i,j}\leq Cij$, which is implied by~\eqref{sub}. \hfill$\qed$


\section*{Acknowledgement}
The research leading to this paper was partially funded by the french ``ANR blanche'' project Kibord: ANR-13-BS01-0004. K.F. was partially supported by NAWI Graz and acknowledges the kind hospitality of the ENS Cachan.

\section*{Appendix}

This section is devoted to the proof of the existence of solutions of dual problems such as 
\begin{equation*}
\left\{
\begin{aligned}
& \partial_t v + M\Delta_x v = f \quad &\text{on}& \ \Omega_T, \\
& \nabla_x v\cdot \nu = 0 \quad &\text{on}& \ ]0,T[\times\partial\Omega, \\
&  v(T,\cdot) = 0 \quad &\text{on}& \ \Omega,
\end{aligned}
\right.
\end{equation*}
when $f$ lies in some $L^q(\Omega_T)$, provided that there exist
constants $0<a\le b$ (sufficiently close from one another in the sense of hypothesis~\eqref{hyp:closeness_bis}) such that $a\leq M\leq b$. We emphasise that $M := M(t,x)$ is not assumed to be continuous. 
\par
Note that with the change of time variable $\tau=T-t$, the above dual problem becomes
a forward heat equation with homogeneous initial data:
\begin{equation}\label{ne45}
\left\{
\begin{aligned}
& \partial_t v - M\Delta_x v = -f \quad &\text{on}& \ \Omega_T, \\
& \nabla_x v\cdot \nu = 0 \quad &\text{on}& \ ]0,T[\times\partial\Omega, \\
&  v(0,\cdot) = 0 \quad &\text{on}& \ \Omega.
\end{aligned}
\right.
\end{equation}

\noindent In the case of parabolic equations in divergence form, i.e. when
\begin{equation}
\label{eq:para_div}
\partial_t v - \diverg_x \left(A(t,x)\nabla_x v\right) = f,
\end{equation}
it is well known that the strict ellipticity property
\begin{equation*}
\xi\transp A(t,x)\, \xi \geq \lambda\, \vert \xi\vert^2,\quad \lambda >0, \quad\forall (t,x)\in \Omega_T,\quad \forall \xi\in\R^N,
\end{equation*}
guarantees the existence of weak solutions even if $A$ is only in $L^{\infty}$, see e.g. \cite{LadSolUra68}. 
\par
However, we are here interested in parabolic equations in non-divergence form, i.e.
\begin{equation}
\label{eq:para_non_div}
\partial_t v - \sum_{i,j} \,a_{i,j}(t,x)\, \partial^2_{i,j} v = f,
\end{equation}
 and we recall that when $A(t,x)=\left(a_{i,j}(t,x)\right)_{i,j}$ is not smooth, 
 the two formulations \eqref{eq:para_div} and \eqref{eq:para_non_div} are not equivalent. 
\par
Unfortunately, the existence theory for parabolic equations with discontinuous coefficients is much less developed in the non-divergence case \eqref{eq:para_non_div}
than in the divergence case (\ref{eq:para_div}), and some additional assumptions on $A$ (that is, stronger than strict ellipticity) are needed. 

One class of available results for parabolic equations in non-divergence form assumes coefficients which are VMO in at least sufficiently many of the time/space variables, see e.g. \cite{Kry,HHH} for references which consider $f\in L^q(\Omega_T)$ with $1<q<2$.

In \cite{Mau02}, equation~\eqref{eq:para_non_div} is treated as a perturbation of the standard heat equation
\begin{equation*}
\partial_t v - \Delta_x v = f ,
\end{equation*} 
and existence of a unique weak solution is then proven under the so called \textit{Cordes condition}
\begin{equation}
\label{hyp:cordes}
\left\Vert  \frac{\sum_{1\leq i,j \leq N}a_{i,j}^2+1}{\Bigl(\sum_{1\leq i \leq N}a_{i,i}+1\Bigr)^2}  
\right\Vert_{L^{\infty}(\Omega_T)} < \frac{1}{N}, 
\end{equation}
which in some sense measures how far $A$ is from the identity matrix and explicitly involves the dimension $N$.

The equation appearing in (\ref{ne45}) is less general than the one treated in \cite{Mau02},
 since we only consider matrices $A$ of the form
\begin{equation*}
a_{i,j}(t,x)=\delta_{i,j} M(t,x),
\end{equation*}
but unfortunately
 assumption~\eqref{hyp:cordes} may not be satisfied. We can however adapt the techniques of  \cite{Mau02}
to get a proof of existence (and uniqueness) under our assumptions on $M$.
 The main idea consists in considering $\partial_t - M\,\Delta_x$ as a perturbation of $\partial_t - m\,\Delta_x$ (instead of $\partial_t - \Delta_x$), where $m$ can be seen as a mean value of $M$. This is the content of the following:

\begin{proposition}\label{prop:exist_para_disc} 
Let $\Omega$ be a bounded smooth subset of $\R^N$ and
consider $M:\Omega_T\to\R_+$ satisfying $a\leq M \leq b$ for some $a,b>0$, and $f\in L^q(\Omega_T)$. Assume that the closeness condition \eqref{hyp:closeness_bis} holds.

Then, there exists a unique $u\in L^q(]0,T[;  W^{2,q} (\Omega) ) \cap W^{1,q} (]0,T[; L^q(\Omega))$ such that
\begin{equation}
\label{eq:syst_para_disc}
\left\{
\begin{aligned}
& \partial_t u - M\Delta_x u = f \quad &\text{on}& \ \Omega_T, \\
& \nabla_x u\cdot \nu = 0 \quad &\text{on}& \ ]0,T[\times\partial\Omega, \\
&  u(0,\cdot) = 0 \quad &\text{on}& \ \Omega.
\end{aligned}
\right.
\end{equation}
\end{proposition}
\begin{proof}
We first rewrite the equation as a perturbation of a heat equation with constant diffusion coefficient:
\begin{equation*}
\partial_t u - m\Delta_x u = -(m-M)\Delta_x u + f,
\end{equation*} 
where $m=\frac{a+b}{2}$. Then we introduce the space 
\begin{equation*}
Z^q := \left\{ v \in L^q(]0,T[;  W^{2,q} (\Omega) ) \cap W^{1,q} (]0,T[; L^q(\Omega)),\ \nabla_x v\cdot \nu = 0 ,\ v(0,\cdot) = 0 \right\},
\end{equation*}
and the operator $F$ defined on $Z^q$, which associates
 to each $v\in Z^q$ the unique solution $Fv\in Z^q$ of
\begin{equation*}
\left\{
\begin{aligned}
& \partial_t (Fv) - m\,\Delta_x (Fv) = -(m-M)\,\Delta_x v + f \quad &\text{on}& \ \Omega_T, \\
& \nabla_x (Fv)\cdot \nu = 0 \quad &\text{on}& \ [0,T]\times\partial\Omega, \\
&  (Fv)(0,\cdot) = 0 \quad &\text{on}& \ \Omega.
\end{aligned}
\right.
\end{equation*}
Proving Proposition~\ref{prop:exist_para_disc} is now equivalent to showing
 the existence of a unique fixed point for $F$. We  endow $Z^q$ with the norm
\begin{equation*}
\left\Vert v\right\Vert_{Z^q_m(\Omega_T)} = \left( \int_{\Omega_T} \left\vert \partial_t v \right\vert^q +m^q \int_{\Omega_T} \left\vert \Delta_x v \right\vert^q \right)^{\frac{1}{q}},
\end{equation*} 
which makes $Z^q$ a Banach space. Note that this is indeed a norm on $Z^q$ thanks to Calderon-Zygmund inequality:
\begin{equation*}
\int\limits_{\Omega_T} \left\vert D^2_x v \right\vert^q \leq C \int\limits_{\Omega_T} \left\vert \Delta_x v \right\vert^q.
\end{equation*}
We now show that $F$ is a contraction on $\left(Z^q,\left\Vert \cdot \right\Vert_{Z^q_m}\right)$, which will yield the existence of a unique fixed point by the contraction mapping Theorem. For any $v,w\in Z^q$, we have
\begin{equation*}
\partial_t (Fv-Fw) - m\Delta_x (Fv-Fw) = -(m-M)\Delta_x (v-w),
\end{equation*}
so that remembering Definition~\ref{def:C_mq},
\begin{align*}
\left\Vert Fv-Fw\right\Vert_{Z^q_m(\Omega_T)} & \leq \frac{b-a}{2} \,\constC_{m,q} \left\Vert \Delta_x(v-w)\right\Vert_{L^q(\Omega_T)} 
 \leq \frac{b-a}{2}\, \frac{\constC_{m,q}}{m} \left\Vert v-w \right\Vert_{Z^q_m(\Omega_T)}. 
\end{align*}
Thus, thanks to $m=\frac{a+b}{2}$ and assumption~\eqref{hyp:closeness_bis}, $F$ is a contraction.
\end{proof}

\noindent Note that in the Hilbert space case $q=2$, it is easily possible (see e.g. \cite{Pie10}) to obtain an explicit bound on $\constC_{m,2}$, namely $\constC_{m,2}\leq 1$, which shows that assumption~\eqref{hyp:closeness_bis} is always satisfied for $q=2$ (and hence for $q$ sufficiently close to $2$, see Remark \ref{dualL2eps}). This is the content of the following

\begin{lemma}
\label{lem:q=2}
For all $m>0$, we have $\constC_{m,2}\leq 1$, see e.g. \cite{Pie10}.
\end{lemma}
\begin{proof}
By multiplying 
\begin{equation*}
\partial_t v - m\Delta_x v = f
\end{equation*}
once by $\partial_t v$, once by $-m\Delta_x v$, and adding the results, we get
\begin{equation*}
\left(\partial_t v\right)^2 + m^2\left(\Delta_x v\right)^2 -2m\,\partial_t v \Delta_x v = f^2.
\end{equation*}
We now show that $ \int_{\Omega_T} \partial_t v\, \Delta_x v \leq 0$. Integrating by parts 
and using the Neumann boundary conditions and the homogeneous initial data, we see indeed that
\begin{align*}
\int_{\Omega_T} \partial_t v \, \Delta_x v = - \int_0^T\!\! \int_{\Omega} \partial_t \nabla_x v\cdot
 \nabla_x v = -\frac{1}{2} \int_{\Omega} \left\vert \nabla_x v \right\vert^2(T) \leq 0.
\end{align*}
Therefore, we obtain 
$$ 
\int_{\Omega_T} (\partial_t v)^2 + m^2\, \int_{\Omega_T} (\Delta_x v)^2 \le \int_{\Omega_T} f^2,
$$
so that $\constC_{m,2}\leq 1$.
\end{proof}

\end{document}